\documentclass[11pt]{amsart}
\usepackage{url}
\usepackage{amsfonts, mathrsfs}
\usepackage{amssymb}
\usepackage{amsmath}
\usepackage{xr-hyper}
\usepackage{amsthm}
\usepackage[lofdepth,lotdepth]{subfig}
\usepackage{setspace,amssymb,verbatim, enumitem}
\usepackage[margin=.9in]{geometry}

\usepackage{lineno}

\usepackage{prettyref}
\newrefformat{cond}{condition (\ref{#1})}
\newrefformat{conj}{Conjecture \ref{#1}}
\usepackage{graphicx}
\usepackage{rotating}
\usepackage{psfrag}
\usepackage{afterpage}
\usepackage{multirow}
\usepackage[lofdepth,lotdepth]{subfig}
\usepackage{tikz}

\theoremstyle{plain}
\newtheorem{theorem}{Theorem}[section]
\newtheorem{proposition}[theorem]{Proposition}
\newtheorem*{theorem*}{Theorem}

\newtheorem*{corollary*}{Corollary}
\newtheorem{corollary}[theorem]{Corollary}

\newtheorem*{lemma*}{Lemma}
\newtheorem{lemma}[theorem]{Lemma}

\newtheorem{thml}{Theorem}

\theoremstyle{remark}

\newtheorem{remark}[theorem]{Remark}

\newtheorem*{case*}{Case}

\theoremstyle{definition}

\newtheorem*{MNConj}{McKay--Navarro Conjecture}

\usepackage{prettyref}
\newrefformat{sec}{Section \ref{#1}}
\newrefformat{chap}{Chapter \ref{#1}}
\newrefformat{prop}{Proposition \ref{#1}}
\newrefformat{prob}{Problem \ref{#1}}
\newrefformat{rem}{Remark \ref{#1}}
\newrefformat{cor}{Corollary \ref{#1}}
\newrefformat{cond}{Condition \ref{#1}}
\newrefformat{def}{Definition \ref{#1}}
\newrefformat{ex}{Example \ref{#1}}
\newrefformat{lem}{Lemma \ref{#1}}
\newrefformat{eq}{Equation \eqref{#1}}
\newrefformat{con}{condition (\ref{#1})}
\newrefformat{conj}{Conjecture \ref{#1}}
\newrefformat{fig}{Figure \ref{#1}}
\newrefformat{tab}{Table \ref{#1}}
\newrefformat{app}{Appendix \ref{#1}}

\newcommand{\Q}{\mathbb{Q}}

\newcommand{\aut}{\mathrm{Aut}}
\newcommand{\out}{\mathrm{Out}}

\newcommand{\F}{\mathbb{F}}

\newcommand{\syl}{\mathrm{Syl}}
\newcommand{\Syl}{\mathrm{Syl}}

\newcommand{\Irr}{\mathrm{Irr}}

\newcommand{\Ind}{\mathrm{Ind}}
\newcommand{\irr}{\mathrm{Irr}}

\newcommand{\wh}[1]{\widehat{#1}}

\newcommand{\la}{\lambda}

\newcommand{\wt}[1]{\widetilde{#1}}
\newcommand{\bG}{\mathbf{G}}
\newcommand{\bP}{\mathbf{P}}
\newcommand{\bH}{\mathbf{H}}
\newcommand{\bM}{\mathbf{M}}
\newcommand{\bL}{\mathbf{L}}
\newcommand{\bT}{\mathbf{T}}
\newcommand{\bS}{\mathbf{S}}
\newcommand{\bg}[1]{\mathbf{#1}}
\newcommand{\norm}{\mathrm{N}}
\newcommand{\cen}{\mathrm{C}}
\newcommand{\zen}{\mathrm{Z}}
\newcommand{\HC}{\mathrm{R}}

\newcommand{\SL}{\operatorname{SL}}
\newcommand{\GL}{\operatorname{GL}}

\newcommand{\type}[1]{\mathsf{#1}}
\newcommand{\Sp}{\operatorname{Sp}}

\newcommand{\PSp}{\operatorname{PSp}}
\newcommand{\PSL}{\operatorname{PSL}}
\newcommand{\gal}{\mathcal{G}}
\newcommand{\galh}{\mathcal{H}}

\newcommand{\Res}{\mathrm{Res}}
\newcommand{\alt}{\mathfrak{A}}

\title[Inductive McKay--Navarro Conditions]{The Inductive McKay--Navarro Conditions for the Prime 2 and Some Groups of Lie Type}
\author{L. Ruhstorfer}
\address{{Fachbereich Mathematik}, {TU Kaiserslautern}, {67653 Kaiserslautern, Germany}}
\email{ruhstorf@mathematik.uni-kl.de}
\author{A. A. Schaeffer Fry}
\address{{Dept. of Mathematics and Statistics}, {Metropolitan State University of Denver}, {Denver, CO 80217}}
\email{aschaef6@msudenver.edu}

\date{}
\thanks{The authors thank the  Isaac Newton Institute
for Mathematical Sciences in Cambridge and the organizers of the Spring 2020 INI program Groups, Representations, and
Applications: New Perspectives, supported by EPSRC grant EP/R014604/1, where this work began.  The second author also gratefully acknowledges support from the National Science Foundation (Award Nos. DMS-1801156 and DMS-2100912).  
The authors also thank G. Malle for helpful comments on an early preprint, as well as the anonymous referee for their careful reading and comments that improved the exposition.}

\usepackage{hyperref}
\begin{document}
\maketitle

\begin{abstract}

For a prime $\ell$, the McKay conjecture suggests a bijection between the set of irreducible characters of a finite group with $\ell’$-degree and the corresponding set for the normalizer of a Sylow $\ell$-subgroup. Navarro's refinement suggests that the values of the characters on either side of this bijection should also be related, proposing that the bijection commutes with certain Galois automorphisms. Recently, Navarro--Späth--Vallejo have reduced the McKay--Navarro conjecture to certain “inductive” conditions on finite simple groups.  We prove that these inductive McKay--Navarro (also called the inductive Galois--McKay) conditions hold for the prime $\ell=2$ for several groups of Lie type, namely the untwisted groups without nontrivial graph automorphisms.

\vspace{0.25cm}

\noindent \textit{Mathematics Classification Number:} 20C15, 20C33

\noindent \textit{Keywords:} local-global conjectures, characters, McKay conjecture, Galois--McKay, McKay--Navarro conjecture, finite simple groups, Lie type, Galois automorphisms

\end{abstract}

\section{Introduction} 

Many of the current problems in the character theory of finite groups are motivated by the celebrated McKay conjecture.  This conjecture states that for a prime $\ell$ dividing the size of a finite group $G$, there should be a bijection between the set $\irr_{\ell'}(G)$ of irreducible characters of degree relatively prime to $\ell$ and the corresponding set $\irr_{\ell'}(\norm_G(P))$ for the normalizer of a Sylow $\ell$-subgroup $P$ of $G$.   Navarro later posited that this bijection should moreover be compatible with the action of certain Galois automorphisms. 

We write $\Q^{\mathrm{ab}}$ for the field generated by all roots of unity in $\mathbb{C}$, and let $\gal:=\mathrm{Gal}(\Q^{\mathrm{ab}}/\Q)$ be the corresponding Galois group. For a prime $\ell$, let $\galh_\ell$ be the subgroup of $\gal$ consisting of those $\tau\in\gal$ such that there is an integer $f$ such that $\tau(\xi)=\xi^{\ell^f}$ for every root of unity $\xi$ of order not
divisible by $\ell$.  The following is Navarro's refinement of the McKay conjecture, sometimes also referred to as the Galois--McKay conjecture.

\begin{MNConj}[\cite{navarro2004}]
Let $G$ be a finite group and let $\ell$ be a prime.  Let $P\in\syl_\ell(G)$.  Then there exists an $\galh_\ell$-equivariant bijection between the sets $\irr_{\ell'}(G)$ and $\irr_{\ell'}(\norm_G(P))$.
\end{MNConj}

 A huge step toward the proof of this conjecture was recently accomplished in \cite{NSV}, where Navarro--Sp{\"a}th--Vallejo proved a reduction theorem for the McKay--Navarro conjecture.  There, they show that the conjecture holds for all finite groups if every finite simple group satisfies certain ``inductive" conditions.  
 
 These inductive McKay--Navarro conditions build upon the inductive conditions for the usual McKay conjecture, which was reduced to simple groups in \cite{IsaacsMalleNavarroMcKayreduction}.  Much work has been done toward proving the original inductive McKay conditions, and they were even completed for the prime $\ell=2$ in \cite{MS16}.  The focus of the current paper is to work toward expanding that result to the case of the inductive McKay--Navarro conditions.

In \cite{ruhstorfer}, the first author proves that the simple groups of Lie type satisfy the inductive McKay--Navarro conditions for their defining prime, with a few low-rank exceptions that were settled in \cite{johansson}.  In \cite{SF20}, the second author shows that the maps used in \cite{malle07, malle08, MS16} to prove the original inductive McKay conditions for the prime $\ell=2$ are further $\galh_\ell$-equivariant, making them promising candidates for the inductive McKay--Navarro conditions.  

Here we expand on the work of \cite{SF20} to complete the proof that when $\ell=2$, the inductive McKay--Navarro (also called inductive Galois--McKay) conditions from \cite{NSV} are satisfied for the untwisted groups of Lie type that do not admit graph automorphisms. 

\begin{thml}\label{thm:iMN}
  The inductive McKay--Navarro conditions \cite[Definition 3.1]{NSV} hold for the prime $\ell=2$ for simple groups 
 $\PSp_{2n}(q)=\type{C}_n(q)$ with $n\geq 1$;
 $\type{B}_n(q)$ with $n\geq 3$ and $(n,q)\neq (3,3)$;
 $\type{G}_2(q)$ with $q\neq 3^f$; 
 $\type{F}_4(q)$;
 $\type{E}_7(q)$; and
 $\type{E}_8(q)$,
where $q$ is a power of an odd prime.
\end{thml}


We remark that Theorem \ref{thm:iMN} gives the first result in which a series of groups of Lie type is shown to satisfy the inductive McKay--Navarro conditions for a non-defining prime. Given the results of \cite{SF20}, the bulk of what remains to prove Theorem \ref{thm:iMN} is to study the behavior of extensions of odd-degree characters of groups of Lie type to their inertia group in the automorphism group, and similarly for the local subgroup used in the inductive conditions.  To do this, we utilize work by Digne and Michel \cite{DM94} on Harish-Chandra theory for disconnected groups and work of Malle and Sp{\"a}th \cite{MS16} that tells us that, in our situation, odd-degree characters lie in very specific Harish-Chandra series.

The paper is structured as follows. In Section \ref{sec:HCdiscon}, we discuss a generalization of some aspects of Harish-Chandra theory to the case of disconnected reductive groups, which will help us understand the extensions of the relevant characters in our situation.  In Section \ref{sec:principalseries}, we use this to make some observations specific to the principal series representations of groups of Lie type that will be useful in our extensions. In Section \ref{sec:iMNconds}, we make some additional observations regarding extensions and the inductive McKay--Navarro conditions. In Section \ref{sec:typeC}, we complete the proof of Theorem \ref{thm:iMN} in the case of type $\type{C}$, and in Section \ref{sec:remaining}, we finish the proof in the remaining cases.

\subsection{Notation}
We introduce some standard notation that we will use throughout. For finite groups $X\lhd Y$, if $\psi\in\irr(X)$, we use the notation $\Irr(Y|\psi)$ to denote the set of $\chi\in\irr(Y)$ such that $\langle \psi, \Res_X^Y \chi\rangle\neq 0$.  If $\chi\in\irr(Y)$, we use $\irr(X|\chi)$ to denote the set of $\psi\in\irr(X)$ such that $\langle \psi, \Res_X^Y \chi\rangle\neq 0$. Further, if a group $A$ acts on $\irr(X)$ and $\psi\in \irr(X)$, we write $A_\psi$ for the stabilizer in $A$ of $\psi$.

\section{Harish-Chandra Theory and Disconnected Groups}\label{sec:HCdiscon}

Let $\bG^\circ$ be a connected reductive group defined over an algebraic closure of $\mathbb{F}_p$ and let $F: \bG^\circ \to \bG^\circ$ a Frobenius endomorphism defining an $\mathbb{F}_q$-structure. Consider a quasi-central automorphism $\sigma: \bG^\circ \to \bG^\circ$ commuting with the action of the Frobenius $F$ as in \cite[Definition-Theorem 1.15]{DM94}. We form the reductive group $\bG:=\bG^\circ \rtimes \langle \sigma \rangle$. Let $\bP^\circ$ be a $\sigma$-stable parabolic subgroup of $\bG^\circ$ with $\sigma$-stable Levi complement $\bL^\circ$ and unipotent radical $\mathbf{U}$. We can then consider the parabolic subgroup $\bP:=\bP^\circ \langle \sigma \rangle$ with Levi subgroup $\bL:=\bL^\circ \langle \sigma \rangle$ in $\bG$. In the following, we assume that $(\bL,\bP)$ is $F$-stable such that we can consider the generalized Harish-Chandra induction
$$\HC_{\bL}^{\bG}:=\mathrm{Ind}_{\bP^F}^{\bG^F} \mathrm{Infl}_{\bL^F}^{\bP^F}: \mathbb{Z} \Irr(\bL^F) \to \mathbb{Z}\Irr(\bG^F).$$
We note that, according to \cite[Theorem 3.2]{DM94}, this map satisfies a Mackey formula in the sense of \cite[Definition 3.1]{DM94}. We call a character $\delta \in \Irr(\bL^F)$ cuspidal if any (hence every) irreducible constituent $\delta^\circ \in \Irr((\bL^\circ)^F)$ of $\delta|_{(\bL^\circ)^F}$ is cuspidal.

\begin{lemma}\label{cuspidal}
	Let $\bL$ and $\bM$ be two $F$-stable Levi subgroups of $F$-stable parabolic subgroups of $\bG$ as above.
	If $\delta \in \Irr(\bL^F)$ is cuspidal and $\bM$ is a proper Levi subgroup of $\bL$, then we have ${}^\ast \HC_{\bM}^{\bL} \delta =0$. 
\end{lemma}

\begin{proof}
	By \cite[Corollary 2.4]{DM94} we have 
	$$\Res_{(\bM^\circ)^F}^{\bM^F} {}^\ast \HC_{\bM}^{\bL} \delta={}^\ast \HC_{\bM^\circ}^{\bL^\circ} \Res_{(\bL^\circ)^F}^{\bL^F} \delta=0,$$ where the last equality follow from the fact that all irreducible constituents of $\Res_{(\bL^\circ)^F}^{\bL^F} \delta$ are cuspidal. Since ${}^\ast \HC_{\bM}^{\bL} \delta \in \mathbb{N} \Irr(\bM^F)$ and $\Res_{(\bM^\circ)^F}^{\bM^F} {}^\ast \HC_{\bM}^{\bL} \delta=0$ it follows that ${}^\ast \HC_{\bM}^{\bL} \delta=0$.
\end{proof}

\begin{proposition}\label{scalar product}
Let $\delta \in \Irr(\bL^F)$ and $\eta \in \Irr(\bL^F)$ be two extensions of the same cuspidal character $\delta^\circ \in \Irr((\bL^\circ)^F)$. Suppose that $\delta^\circ$ extends to its inertia group in $\mathrm{N}_{\bG^F}(\bL^\circ)$ and $\sigma$ centralizes $N:=\mathrm{N}_{(\bG^\circ)^F}(\bL^\circ)/(\bL^\circ)^F$. Then we have
$$\langle \HC_\bL^\bG \delta, \HC_\bL^\bG \eta \rangle=|N_{\delta^\circ}| \langle \delta, \eta \rangle.$$
\end{proposition}

\begin{proof}
	We loosely follow the proof of \cite[Proposition 3.2]{DM15}. 
By definition, we have
$$\langle \HC_\bL^\bG \delta, \HC_\bL^\bG \eta \rangle= \frac{1}{r} \sum_{i=1}^r \langle R_{\bL^\circ \sigma^i}^{\bL^\circ \sigma^i}(\delta), R_{\bL^\circ \sigma^i}^{\bG^\circ \sigma^i}(\eta) \rangle_{ (\bG^\circ)^F \sigma^i}$$
for  $\sigma$ of order $r$.
According to the Mackey formula from \cite[Theorem 3.2]{DM94} and the proof of \cite[Lemma 6.5]{DM} we obtain
$$\langle R_{\bL^\circ \sigma^i}^{\bG^\circ \sigma^i}(\delta), R_{\bL^\circ \sigma^i}^{\bG^\circ \sigma^i}(\eta) \rangle_{ (\bG^\circ)^F \sigma^i} = 
\sum_{x \in \mathrm{N}_{((\bG^{\sigma^i})^\circ)^F}((\bL^{\sigma^i})^\circ)^F/ ((\bL^{\sigma^i})^\circ)^F} \langle {}^x\delta, \eta \rangle_{(\bL^\circ)^F \sigma^i}.$$
Note that $0= \langle {}^x \delta, \eta \rangle_{(\bL^\circ)^F \sigma^i}$ unless ${}^x \delta^\circ=\delta^\circ$. 
Hence, $x \in \mathrm{N}_{(\bG^\circ)^F}(\bL^\circ)_{\delta^\circ}$. Since $\delta^\circ$ extends to its inertia group in $\mathrm{N}_{\bG^F}(\bL^\circ)$ we deduce by \cite[Lemma 2.11]{spath12} that ${}^x \delta= \delta$. By \cite[Proposition 1.39]{DM94} and \cite[Corollary 1.25(ii)]{DM94}, we have $$\mathrm{N}_{((\bG^{\sigma^i})^\circ)^F}((\bL^{\sigma^i})^\circ)^F (\bL^\circ)^F=\mathrm{N}_{(\bG^\circ)^F}(\bL^\circ \sigma^i).$$
Since $\sigma$ centralizes $N$ by assumption we deduce that
$$N \cong \mathrm{N}_{((\bG^{\sigma^i})^\circ)^F}((\bL^{\sigma^i})^\circ)^F/ ((\bL^{\sigma^i})^\circ)^F.$$
Therefore, the scalar product evaluates to
$$\langle \HC_\bL^\bG \delta, \HC_\bL^\bG \eta \rangle=\frac{1}{r} |N_{\delta^\circ}| \sum_{i=1}^r \langle \delta,\eta \rangle_{(\bL^\circ)^F \sigma^i}= |N_{\delta^\circ}| \langle \delta, \eta \rangle.$$
This completes the proof.  
\end{proof}

\begin{proposition}\label{uniquev2}
Keep the assumptions of the previous proposition. Let $\chi^\circ \in \Irr((\bG^\circ)^F)$ be a character in the Harish--Chandra series of $\delta^\circ$.
	Then for every character $\chi \in \Irr(\bG^F \mid \chi^\circ)$, there exists a unique $\delta \in \Irr(\bL^F \mid \delta^\circ)$ such that $\langle \HC_{\bL}^{\bG} \delta, \chi \rangle \neq 0$. 
\end{proposition}

\begin{proof}
	Since Harish-Chandra induction commutes with induction we have
	$$\langle \chi, \HC_{\bL}^{\bG} \Ind_{(\bL^\circ)^F}^{\bL^F} \delta^\circ \rangle= \langle \chi, \Ind_{(\bG^\circ)^F}^{\bG^F}  \HC_{\bL^\circ }^{\bG^\circ} \delta^\circ \rangle = \langle  \Res_{(\bG^\circ)^F}^{\bG^F}  \chi,  \HC_{\bL^\circ }^{\bG^\circ} \delta^\circ \rangle,$$
	where the last equality follows from Frobenius reciprocity. From this it follows that there exists some $\delta \in \Irr(\bL^F \mid \delta^\circ)$ such that $\langle \HC_{\bL}^{\bG} \delta, \chi \rangle \neq 0$. The uniqueness follows from Lemma \ref{scalar product}.

\end{proof}


\begin{corollary}\label{cor:hcbij}
	Assume that we are in the situation of the previous proposition.
Then we have a $\Irr(\bG^F/ (\bG^\circ)^F) \times \gal_{\delta^\circ}$-equivariant {bijection} 
	$$\Irr(\bG^F \mid \chi^\circ) \to \Irr(\bL^F \mid \delta^\circ), \chi\mapsto\delta.$$
In particular, $\chi$ is an extension of $\chi^\circ$ and we have 
$$\langle \chi, \HC_{\bL}^{\bG}  \delta \rangle = \langle \chi^\circ, \HC_{\bL^\circ}^{\bG^\circ}  \delta^\circ \rangle.$$ 	
\end{corollary}

\begin{proof}
The map itself follows from Proposition \ref{uniquev2}, since the fact that $\bG^F=(\bG^\circ)^F\bL^F$ and $(\bG^\circ)^F\cap \bL^F=(\bL^\circ)^F$ yields that $\Res_{(\bG^\circ)^F}^{\bG^F}\Ind_{\bL^F}^{\bG^F} = \Ind_{(\bL^\circ)^F}^{(\bG^\circ)^F}\Res_{(\bL^\circ)^F}^{\bL^F}$ (see e.g. \cite[Problem (5.2)]{isaacs}), implying that given $\delta\in\Irr(\bL^F \mid \delta^\circ)$, there exists $\chi\in\Irr(\bG^F \mid \chi^\circ)$ with $\langle \HC_{\bL}^{\bG}\delta, \chi\rangle\neq 0$.  
For $\lambda \in \Irr(\bG^F / (\bG^\circ)^F)$ we have $\lambda \HC_{\bL}^{\bG} \delta=\HC_{\bL}^{\bG} \lambda \delta$, by e.g. \cite[Problem 5.3]{isaacs}, so the compatibility with linear characters follows from the uniqueness statement in Proposition \ref{uniquev2}. As $\HC_{\bL}^{\bG}$ is equivariant with respect to Galois-automorphisms we similarly obtain that the map is $\gal_{\delta^\circ}$-equivariant. Altogether this implies that $\chi^\circ$ extends to $\bG^F$ and the desired multiplicity formula.
\end{proof}

The previous corollary allows us in principle to compute the action of the Galois group $\gal_{\delta^\circ}$ on constituents of $\HC_{\bL}^{\bG}(\delta)$. However, it could happen that $\gal_{\chi^\circ}$ is not contained in  $\gal_{\delta^\circ}$.
The following lemma solves this problem using the assumption that $\sigma$ centralizes the relative Weyl group of $\bL$. 

\begin{lemma}\label{lem: ext}
{Keep the notation and assumption from Proposition \ref{uniquev2}.}
\begin{enumerate}
	\item 
We have $\norm_{\bG^F}(\bL)=\{ g\in \norm_{(\bG^\circ)^F}(\bL^\circ) \mid \sigma(g)g^{-1} \in (\bL^\circ)^F \} \langle \sigma \rangle$.
\item 
Suppose that $\tau \in \gal$ stabilizes $\chi^\circ$. If $\sigma$ centralizes $\norm_{(\bG^\circ)^F}(\bL^\circ)/ (\bL^\circ)^F$ and stabilizes $\chi^\circ$, then we have $\chi^{\tau}=\chi \lambda$ for some
$\lambda \in \Irr(\bG^F)$.
\end{enumerate}
\end{lemma}

\begin{proof}
	Part (1) is an easy computation. Let us prove part (2).	Since $\tau \in \gal$ stabilizes $\chi^\circ$ we have $(\delta^\circ)^{\tau x}=\delta^\circ$ for some $x \in \norm_{(\bG^\circ)^F}(\bL^\circ)$ by \cite[Lemma 6.5]{DM}. 
	By assumption, $\sigma$ centralizes $\norm_{(\bG^\circ)^F}(\bL^\circ)/ (\bL^\circ)^F$ and so we have we have $\sigma(x) x^{-1} \in {\bL^\circ}^F$. 
	From this we deduce that $x$ normalizes $\bL^F$ and thus there exists some linear character $\lambda$ {of $\bL^F/(\bL^\circ)^F\cong {\bG}^F/(\bG^\circ)^F$} such that $\delta^{\tau x}=\delta \lambda$. We deduce that 
	$$0 \neq \langle \chi^{\tau x} , \HC_\bL^\bG \delta^{\tau x} \rangle =\langle \chi^{\tau x}, \HC_\bL^\bG \delta \lambda \rangle = \langle \chi^{\tau x} \lambda^{-1}, \HC_\bL^\bG \delta \rangle.$$
	Since both $\chi^{\tau x} \lambda^{-1}$ and $\chi$ cover $\chi^\circ,$ we obtain $\chi^{\tau x} \lambda^{-1}=\chi^\tau \lambda^{-1} = \chi$ by Corollary \ref{cor:hcbij}.
\end{proof}

\subsection{Descent of scalars}\label{sec:descentscalars}

Let us from now on assume that $\bG$ is a connected reductive group with Frobenius endomorphism $F: \bG \to \bG$. If $\mathbf{H}$ is an $F$-stable subgroup of $\bG$  then we write $H$ for the set $\mathbf{H}^F$ of $F$-fixed points.

We assume that $F_0: \bG \to \bG$ is a Frobenius endomorphism satisfying $F_0^r=F$ for some positive integer $r$.
For an $F_0$-stable subgroup $\mathbf{H}$ of $\bG$, we consider $\underline{\bH}= \bH^r$ with Frobenius endomorphism $F_0 \times \dots \times F_0: \underline{\bH} \to \underline{\bH}$ which we also denote by $F_0$ and the permutation
$$\pi: \underline{\bH} \to \underline{\bH}, \, (g_1,\dots,g_r) \mapsto (g_2,\dots,g_r,g_1).$$
The projection map $\mathrm{pr}: \underline{\bH} \to \bH$ onto the first coordinate then yields an isomorphism $\underline\bH^{\pi F_0} \cong \bH^{F}$ and the automorphism $\pi \in \mathrm{Aut}(\underline{\bG}^{ \pi F_0})$ corresponds to $F_0^{-1} \in \mathrm{Aut}(\bG^F)$ under this isomorphism. Moreover, $\mathrm{pr}$ induces a bijection $\mathrm{pr}^\vee: \Irr(\bH^F) \to \Irr(\underline \bH^{\pi F_0})$ on the level of characters.

Note that $\pi$ is a quasi-central automorphism of $\bG$. We fix an $F_0$-stable parabolic subgroup $\bP$ of $\bG$ with $F_0$-stable Levi subgroup $\bL$. We assume for the remainder of this section that $F_0$ centralizes the relative Weyl group $\mathrm{N}_{G}(\bL)/L$. This implies that the projection map $\mathrm{pr}$ maps
$\norm_{\underline{\bG}^{\pi F_0}}(\underline{\bL} \langle \pi \rangle)$ onto $\norm_{G}(\bL)$.

For any $F_0$-stable subgroup $H$ of $G$ we set $\hat H:=H \langle F_0 \rangle$. We define a generalized Harish-Chandra induction
$$\HC_{\hat L}^{\hat G}:=\Ind_{\hat P}^{\hat G} \circ \mathrm{Infl}_{\hat L}^{\hat P},$$
and by construction we obtain $\mathrm{pr}^\vee \HC_{\hat L}^{\hat G}=\HC_{\underline{\bL} \langle \pi \rangle }^{\underline{\bG} \langle \pi \rangle} \mathrm{pr}^\vee.$
With this notation, we can reformulate Proposition \ref{uniquev2}:

\begin{proposition}\label{prop:uniquedescent}
	Suppose that $\delta \in \Irr(L)$ is an $F_0$-invariant cuspidal character that extends to a character of $\norm_{G}(\bL)_{\delta} \langle F_0 \rangle$, and let $\chi \in \Irr(G)$ be an $F_0$-invariant character in the Harish--Chandra series of $\delta$. Then for $\hat \chi \in \Irr(\hat G \mid \chi)$, there exists a unique $\hat \delta \in \Irr(\hat L \mid \delta)$ such that $\langle \HC_{\hat L}^{\hat G} \hat \delta, \hat \chi \rangle \neq 0$.  {If $\chi$ is $F_0$-invariant, this yields a bijection $\irr(\hat{G}\mid\chi)\rightarrow  \irr(\hat{L}\mid \delta)$.}
\end{proposition}

\begin{proof}
	To prove this, consider Proposition \ref{uniquev2} in the case of the $\pi$-invariant character $\delta_0:=\mathrm{pr}^\vee(\delta) \in \Irr(\underline{\bL}^{\pi F_0})$ which is a cuspidal character of $\underline \bL^{\pi F_0} \cong \bL^F$. Using the fact that $\norm_{(\underline{\bG} \langle \pi \rangle)^{\pi F_0}}(\underline{\bL} \langle \pi \rangle) \cong \norm_{G}(\bL) \langle F_0 \rangle$ via $\mathrm{pr}$ we see that the assumptions of Proposition \ref{uniquev2} are all satisfied. The commutation of Harish-Chandra induction with $\mathrm{pr}^\vee$ then yields the result.
\end{proof}

In particular, we obtain the following consequence.

\begin{corollary}\label{cor:galhc}
	Keep the notation and assumption from Proposition \ref{prop:uniquedescent} and assume that $\chi$ is $F_0$-invariant. Suppose that $\tau \in \gal$ stabilizes $\chi$. Then we have $\hat \chi^{\tau}=\hat \chi \lambda$ for some $\lambda \in \Irr(G \langle F_0 \rangle /G)$.  {In particular,  $\la$ is such that $\delta^{\tau x}=\delta$ and $\wh{\delta}^{\tau x}=\wh{\delta}\la$ for some $x \in \norm_{G}(\bL)$.}
\end{corollary}

\section{Principal Series}\label{sec:principalseries}
In the situation of Theorem \ref{thm:iMN}, most of the characters that we are concerned with will lie in the principal series, so we present here some consequences of the previous section in that situation. 
We begin with the following, which will be useful on the ``local" side.  

\begin{lemma}\label{lem:localextend}

Let $\hat Y=YA$ be a finite group with $Y\lhd \hat Y$, $A\leq \hat Y$, and $Y\cap A=\{1\}$. Let $X\lhd Y$ be an $A$-stable normal subgroup and let $\theta\in\irr(X)$ be $A$-invariant and extend to $\hat Y_\theta=Y_\theta A$. Let $\psi\in\irr(Y|\theta)$ extend to $\hat Y$. Then the following hold.
\begin{enumerate}
\item Let $\chi\in\irr(Y_\theta|\theta)$ such that $\psi=\Ind_{Y_\theta}^Y(\chi)$ via Clifford correspondence.  Then $\chi$ extends to $\hat Y_\theta$ and $\Ind_{\hat Y_\theta}^{\hat Y}$ induces a bijection $\irr(\hat Y_\theta|\chi)\rightarrow \irr(\hat Y|\psi)$ such that extensions of $\chi$ are mapped to extensions of $\psi$.
\item There is an extension $\hat\theta$ of $\theta$ to $\hat X:=XA\lhd \hat Y_\theta$, and for any such extension, $\Res_{ Y_\theta}^{\hat Y_\theta}$ induces a bijection $\irr(\hat Y_\theta|\hat \theta)\rightarrow \irr( Y_\theta|\theta)$.
\item If $\theta(1)=1$, then there is an extension $\hat \psi\in \irr(\hat Y|\psi)$ such that $\gal_{\hat\psi}\cap\gal_{\theta}=\gal_{\psi}\cap\gal_{\theta}$.
\item Continue to assume that $\theta(1)=1$, and suppose further that $|A|$ is relatively prime to the order $o(\theta)$ of $\theta$.  Let $\Gamma\leq\aut(\hat Y)$ preserve $\hat X$ and  $Y$ (and hence also $ X=\hat X\cap Y$).   Then the extension $\hat \psi$ from (3) satisfies $\Gamma_{\hat\psi}\cap\Gamma_{\theta}=\Gamma_{\psi}\cap\Gamma_{\theta}$.
\end{enumerate}
\end{lemma} 
\begin{proof}
(1) Let $\hat\psi\in\irr(\hat Y)$ be an extension of $\psi$. In particular, note that $\hat\psi\in\irr(\hat Y|\theta)$, and hence by Clifford correspondence, there is a unique $\hat\chi\in\irr(\hat Y_\theta|\theta)$ such that $\hat\psi=\Ind_{\hat Y_\theta}^{\hat Y}(\hat\chi)$. Now, since $Y\cap \hat Y_\theta=Y_\theta$ and $\hat Y=\hat Y_\theta Y$, we have $\psi=\Res_Y^{\hat Y} (\hat\psi)
=\Res_Y^{\hat Y}\Ind_{\hat Y_\theta}^{\hat Y}(\hat\chi)
=\Ind_{Y_\theta}^{Y}\Res_{Y_\theta}^{\hat Y_\theta}(\hat\chi)$ (see e.g. \cite[Problem (5.2)]{isaacs}). This forces $\Res_{Y_\theta}^{\hat Y_\theta}(\hat\chi)$ to be irreducible and $\chi=\Res_{Y_\theta}^{\hat Y_\theta}(\hat\chi)$ by the uniqueness of $\chi$ under Clifford correspondence, since both lie over $\theta$.
Now, identifying $A\cong\hat Y_\theta/Y_\theta\cong \hat Y/Y$, we have by Gallagher's theorem that $\Irr(\hat Y_\theta|\chi)=\{\hat \chi\gamma\mid \gamma\in \Irr(A)\}$ and $\Irr(\hat Y|\psi)=\{\hat \psi\gamma\mid \gamma\in \Irr(A)\}$.  Note that under our identification, we have $\Ind_{\hat Y_\theta}^{\hat Y}(\hat\chi\gamma)=\hat\psi\gamma$, as the identification $\Irr(\hat Y/Y)=\Irr(\hat Y_\theta/Y_\theta)$ is via restriction (see e.g. \cite[Problem (5.3)]{isaacs}).

(2) Let $\hat\theta:=\Res^{\hat Y_\theta}_{\hat X}(\hat \Lambda(\theta))$, where $\hat\Lambda(\theta)$ is an extension of $\theta$ to $\hat Y_\theta$, which exists by assumption.  Then $\hat\theta$ is an extension of $\theta$ to $\hat X$. Again we use Gallagher's theorem to achieve the stated bijection, noting that $\hat Y_\theta/\hat X\cong Y_\theta/X$, with their characters identified via restriction.

(3) Certainly, we have $\gal_{\hat\psi}\cap\gal_{\theta}\subseteq\gal_{\psi}\cap\gal_{\theta}$, and we wish to prove that there exists such a $\hat\psi$ satisfying the converse inclusion. Let $\tau\in\gal_{\psi}\cap\gal_{\theta}$.  Note that the action of $\tau$ commutes with induction, so that letting $\chi$ be as in (1), we have $\chi^\tau=\chi$  by the uniqueness of Clifford correspondence and it suffices by (1) to find an extension $\hat\chi\in\irr(\hat Y_\theta|\chi)$ such that $\hat\chi^\tau=\hat\chi$.  
Now, let $\hat\theta\in\Irr(\hat X)$ be the unique extension of $\theta$ to $\hat X$ such that $A\in \ker(\hat\theta)$.  (Indeed, note that such a character exists since $\theta$ is now assumed to be linear.) Then $\hat\theta^\tau=\hat\theta$, and by (2) there is a unique $\hat\chi\in\irr(\hat Y_\theta|\hat\theta)$ extending $\chi$. But since $\hat\chi^\tau$ also extends $\chi=\chi^\tau$ and lies above $\hat\theta=\hat\theta^\tau$, this forces $\hat\chi^\tau=\hat\chi$.

(4)  We have $\hat\theta$ from the proof of (3) is the unique extension of $\theta$ to $\hat X$ of order $o(\theta)$, by \cite[Corollary (8.16)]{isaacs}, and therefore $\hat\theta^\beta=\hat\theta$ for any $\beta\in \Gamma_\theta$.  
Arguing similarly to part (3) now yields the last statement.
\end{proof}

\begin{remark}
We remark that the conclusion of (4) from Lemma \ref{lem:localextend} would also hold under the assumption that $A$ is $\Gamma$-invariant, but this will not necessarily be the situation for our application.
\end{remark}

Let $\bG$, $F$, and $F_0$ be as in \prettyref{sec:descentscalars} so that $G=\bG^F$ is defined over $\F_q$, 
and let $\bT\leq \bg{B}$ be an $F$-stable maximal torus and Borel subgroup, respectively, in $\bG$ such that $F_0$ centralizes $N/T$, where $T:={\bT}^F$ and $N:=\norm_G(\bT)$.  We will write $\hat T:=T\langle F_0\rangle$, $\hat G:=G\langle F_0\rangle$, and $\hat N:=N\langle F_0\rangle$, as before.

Let $\bG\hookrightarrow\wt{\bG}$ be a regular embedding, as in \cite[(15.1)]{CE04}, and let $\wt{\bT}$ be an $F$-stable maximally split torus of $\wt{\bG}$ such that $\bT=\wt{\bT}\cap\bG$.  Write $\wt{T}:=\wt{\bT}^F$ and $\wt{G}:=\wt{\bG}^F$.  Then $\wt{G}$, and hence $\wt{T}$, induces the so-called diagonal automorphisms on $G$.

For $\delta\in\irr(T)$, the principal series $\mathcal{E}(G, (T, \delta))$ is the Harish-Chandra series of $G$ corresponding to $(T, \delta)$, and is in bijection with the set of irreducible characters of $N_\delta/T$.  We write $\HC_{\bT}^{\bG}(\delta)_{\eta}$ as in \cite[4.D]{MS16} 
for the character in $\mathcal{E}(G, (T, \delta))$ corresponding to $\eta\in\irr(N_{\delta}/T)$.

For finite groups $X\lhd Y$, if every character $\theta\in\irr(X)$ extends to its inertia group $Y_\theta$, then by an extension map with respect to $X\lhd Y$, we mean a map $\Lambda\colon \irr(X)\rightarrow \bigcup_{\theta\in\irr(X)} \irr(Y_\theta)$ such that $\Lambda(\theta)\in \irr(Y_\theta)$ is an extension of $\theta$ for each $\theta\in\irr(X)$.

\begin{lemma}\label{lem:principalextend}
Let $\ell\nmid q$ be a prime and write $\galh:=\galh_\ell$. Let $\tau\in\galh$ and suppose $\delta$ and $\chi\in \mathcal{E}(G, (T, \delta))$ are each invariant under both $F_0$ and $\tau$. Then:
\begin{enumerate}
\item There exists an extension $\hat \chi$ of $\chi$ to $\hat G$ such that $\hat \chi^\tau=\hat \chi$. If further $F_0$ has order prime to $o(\delta)$, then $\hat\chi^{a}=\hat\chi$ for any $a\in\wt{T}_\chi$.
\item If $\Lambda$ is an extension map with respect to $T\lhd N$ such that the map \[\Omega\colon \bigcup_{\delta\in\irr(T)} \mathcal{E}(G, (T, \delta))\rightarrow \irr(N)\] defined by $\HC_{\bT}^{\bG}(\delta)_{\eta}\mapsto \Ind_{N_{\delta}}^{N}(\Lambda(\delta)\eta)$ is $\galh\times \langle F_0\rangle$-equivariant,  
then there also exists an extension $\hat \psi$ of $\psi:=\Omega(\chi)$ to $\hat N$ such that $\hat \psi^\tau=\hat \psi$. 
\end{enumerate}
\end{lemma}

\begin{proof}
Let $\hat \delta\in\irr(\hat T)$ be the extension of $\delta$ to $\hat T$ such that $F_0\in\ker(\hat \delta)$.  
Then in the notation of \prettyref{cor:galhc}, we have $\delta^{\tau x}=\delta=\delta^\tau$, so $x\in\norm_G(T)$ may be chosen to be trivial.  Then $\hat \delta^{\tau x}=\hat \delta^\tau=\hat \delta$ by our choice of $\hat \delta$.  That is, $\lambda=1$ in the notation of \prettyref{cor:galhc}.  Now, let $\hat \chi\in \irr(\hat G|\chi)$  correspond to $\hat \delta$ as in \prettyref{prop:uniquedescent}.  Then \prettyref{cor:galhc} yields $\hat \chi^\tau=\hat \chi$, completing the proof of the first statement of (1).  If $F_0$ has order prime to $o(\delta)$ and $a\in \wt{T}_\chi$, we have $\hat\delta^a=\hat\delta$ since it is the unique extension of $\delta$ to $\hat T$ of order $o(\delta)$, by \cite[Corollary (8.16)]{isaacs}, and $\delta$ is necessarily $\wt{T}$-invariant.  Then $\langle \HC_{\hat T}^{\hat G} \hat \delta, \hat \chi^a \rangle \neq 0$, forcing $\hat\chi^a=\hat\chi$ by Proposition \ref{prop:uniquedescent}. 

Now, in the situation of (2), we have $\psi$ is $\tau$- and $F_0$-invariant by the equivariance of $\Omega$.  Hence the statement follows from Lemma \ref{lem:localextend} applied to $(X,Y,\theta)=(T, N, \delta)$.
\end{proof}

\section{Extensions and the Inductive McKay--Navarro Conditions}\label{sec:iMNconds}


Let $S$ be a nonabelian finite simple group and $\ell$ a prime dividing $|S|$.  Let $H$ be a universal covering group of $S$ and $P\in\Syl_\ell(H)$.  Write $\galh:=\galh_\ell$.
The inductive McKay--Navarro conditions  \cite[Definition 3.1]{NSV} (referred to there as inductive Galois--McKay conditions) require an $\aut(H)_P$-stable subgroup $\norm_H(P)\leq M<H$ and an $\galh\times \aut(H)_P$-equivariant bijection $\Omega\colon \irr_{\ell'}(H)\rightarrow \irr_{p'}(M)$ such that, roughly speaking, 
the projective representations extending any $\chi\in\irr_{\ell'}(H)$ to its inertia subgroup in $H\rtimes \aut(H)_{P, \chi^{\galh}}$ and $\Omega(\chi)$ to its inertia subgroup in $M\rtimes \aut(H)_{P, \chi^{\galh}}$ behave similarly when twisted by an element of $(M\rtimes \aut(H)_{P, \chi^{\galh}}\times \galh)_\chi$.  
(See \cite[Definitions 1.5 and 3.1]{NSV} for the precise definition.)  Here $\chi^\galh$ is the $\galh$-orbit of $\chi$.

In our situation of Theorem \ref{thm:iMN}, the second author has proven in \cite{SF20} that the groups $M$ and bijections $\Omega$ from \cite{malle07, malle08, MS16} for the inductive McKay conditions are indeed $\galh$-equivariant.  Hence here we must study the behavior of the projective representations under the appropriate twists.
We recall the statement of \cite[Lemma 1.4]{NSV}, which makes this idea more precise:

\begin{lemma}\label{function}
	Let $Y$ be a finite group and $X \lhd Y$. Let $\theta \in \Irr (X)$ and assume that
	$\theta^{g\tau }=\theta$ for some $g\in Y$ and
	$\tau \in \gal$.
	Let $\mathcal{P}$ be a projective representation of $Y_\theta$
	associated with $\theta$ with values in $\mathbb{Q}^{\mathrm{ab}}$ and factor set $\alpha$.
	Then $\mathcal{P}^{g\tau}$ is a projective representation
	associated with $\theta$, with factor set $\alpha^{g\tau}(x, y)=\alpha^g(x,y)^\tau$ for $x,y \in Y_\theta$.
	In particular, there exists a unique function $$\mu_{g\tau} : Y_\theta \to K^\times$$
	with $\mu_{g\tau}(1)=1$, constant on cosets of $X$
	such that the projective representation $\mathcal{P}^{g\tau}$ is similar to $\mu_{g\tau} \mathcal{P}$.
\end{lemma}

Given the main results of \cite{MS16, SF20}, our primary remaining obstruction to proving Theorem \ref{thm:iMN} is in working with the characters $\mu_{g\tau}$ as in Lemma \ref{function}.   If $\theta$ extends to $Y_\theta$, then $\mu_{g\tau}$ is just the linear character guaranteed by Gallagher's theorem \cite[Corollary 6.17]{isaacs}.  In this situation, our main task will be to  prove part (iv) of \cite[Definition 1.5]{NSV} for certain triples, which, roughly speaking, will require that the characters $\mu_{g\tau}$ for $g\tau\in (M\rtimes \aut(H)_{P, \chi^{\galh}}\times \galh)_\chi$  are the same on either side of the bijection $\Omega$ when $\chi$ and $\Omega(\chi)$ (or their corresponding projective representations) are extended to their inertia groups in $M\rtimes\aut(H)_{P, \chi^\galh}$ and $H\rtimes \aut(H)_{P, \chi^\galh}$, respectively.
We will assume throughout, without loss of generality, that all linear representations are realized over the field $\mathbb{Q}^{\mathrm{ab}}$. 

\subsection{Further Notation}\label{sec:notn}

Keep the notation for $T=\bT^F$, $G=\bG^F$, and $N$, from Section \ref{sec:principalseries} and further assume that $\bG$ is simple of simply connected type.  If $\zen(G)=1$, write $\wt{\bG}=\bG$, $\wt{G}=G$, $\wt{\bT}=\bT$, and $\wt{T}:=T$, and if $\zen(G)\neq 1$, let $\bG \hookrightarrow \wt{\bG}$ be a regular embedding and let  $\wt{G}:={\wt{\bG}}^F$ and $\wt{T}:=\wt{\bT}^F$ as before.  Note that for an appropriate group $D$ generated by field and graph automorphisms, we have $\wt{G}\rtimes D$ induces $\aut(G)$. 

Let $d:=d_\ell(q)$ be the order of $q$ modulo $\ell$ for $\ell$ odd, or the order of $q$ modulo 4 for $\ell=2$, and suppose that $d\in\{1, 2\}$.  Let $\bS_0$ be a Sylow $d$-torus for $(\bG, F)$ and write $N_0:=\norm_{G}(\bS_0)$ and $\wt{N}_0:=\norm_{\wt{G}}(\bS_0)$.


\subsection{The Characters $\mu_{g\tau}$ in Our Situation}
Let $\wt{X} \in \{ \wt{G}, \wt{N}_0 \}$ and $X:= G \cap \tilde{X}\in\{G, N_0\}$. Suppose that $\psi\in\irr(X)$ is a character such that $(\wt{X} D)_\psi=\wt{X}_\psi D_\psi$ and such that $\psi$ extends to ${X} D_\psi$. Note that since restrictions of irreducible characters from $\wt{G}$ to $G$ are multiplicity-free by the work of Lusztig and the same is true for $\wt{N}_0$ to $N_0$ by \cite[Corollary 3.21]{MS16}, 
 we also have $\psi$ extends to $\wt{X}_\psi$.  


Let $\mathcal{D}$ be a representation affording $\psi$ and let $\mathcal{D}_1$, respectively $\mathcal{D}_2$, be the representation of $\wt{X}_{\psi}$, respectively $X D_\psi$, extending $\mathcal{D}$. 
For $i=1,2$ we let $\psi_i$ be the character of the representation $\mathcal{D}_i$. Then we define $\mathcal{P}$ to be the projective representation of $(\wt{X}D)_{\psi}$ above $\psi$ given by 
$$\mathcal{P}(\wt{g} d):= \mathcal{D}_1(\wt{g}) \mathcal{D}_2(d)$$
for $\wt{g} \in \wt{X}_{\psi}$ and $d \in X D_{\psi}$, as in \cite[Lemma 2.11]{spath12}.

 In this situation, we can now state the following lemma, whose proof is exactly the same as \cite[Lemma 7.2]{ruhstorfer}.

\begin{lemma}\label{compute mu}
Assume we are in the situation above.	Let $y\in \galh \times \norm_{\wt{N}_0D}(X D_{\psi} )$ with $ \psi^y= \psi$. Suppose that $\mu_1 \in \Irr( \tilde{X}_{\psi} /X)$ and $\mu_2 \in \Irr( X D_{\psi} / X)$ are such that $ \psi_i^y = \mu_i \psi_i$ for $i=1,2$. Then there exists an invertible matrix $M$ such that
	$$ \mathcal{P}^y( \tilde{g} d)= \mu_1(\tilde{g}) \mu_2(d) M \mathcal{P}(\tilde{g} d) M^{-1}$$ for all $\tilde{g} \in \tilde{X}_{\psi}$ and $d \in X D_{\psi}$.
\end{lemma}

%

\section{Proof of \prettyref{thm:iMN} for Symplectic Groups}\label{sec:typeC} 

We fix some notation to be used throughout our proof of \prettyref{thm:iMN} in the case of the symplectic groups.  Let $q$ be a power of an odd prime $p$ and let $S:=\PSp_{2n}(q)$ with $n\geq 1$ be one of the simple symplectic groups.   Let $G:=\Sp_{2n}(q)$ and $\wt{G}:=\operatorname{CSp}_{2n}(q)$, so that $G$ is a universal covering group for $S$ unless $(n,q)=(1,9)$, and $\wt{G}\rtimes D$ with $D=\langle F_p\rangle$ induces all automorphisms on $G$.  Here $F_p$ is the field automorphism induced by the map $x\mapsto x^p$ in $\F_q$.  

Note that for $n=1$, we have $S\cong\PSL_2(q)$, $G\cong \SL_2(q)$, and $\wt{G}\cong \GL_2(q)$.  We begin with a lemma to help in this case.  


{
\begin{lemma}\label{lem:SL2}
	Let $G=\SL_2(q)$	and $\tilde{\chi} \in \Irr(\tilde{G})$.  Then there exists a character $\chi \in \Irr(\tilde{G} \mid \tilde{\chi})$ which has an extension $\hat{\chi}$ to $G D_\chi$ such that $\tilde{G} (\tilde{G}\mathcal{G})_\chi=\tilde{G} (\tilde{G}\mathcal{G})_{\hat{\chi}}$. 
\end{lemma}
}
\begin{proof}
	{
	There exists a field automorphism $F_0$ such that $D_\chi=\langle F_0 \rangle$.
	The statement is true for the unipotent characters by Corollary \ref{cor:galhc} as both of them lie in the principal series. So we may assume $\chi$ lies in a rational Lusztig series  $\mathcal{E}(G,s)$ for some semisimple element $1 \neq s$ of the dual group $G^\ast$. It follows that $\cen^\circ_{G^\ast}(s)$ is a maximal torus. In particular, $\chi$ is a regular-semisimple character. As such a character, it is a constituent of a Gelfand--Graev character. Let $\phi_0 \in \Irr(U)$ be a non-trivial $F_0$-stable character, where $U$ is the subgroup of upper unitriangular matrices in $G$. It follows that some $\tilde{G}$-conjugate of $\chi$ is a constituent of multiplicity one of the Gelfand--Graev character $\Gamma:=\Ind_U^G(\phi_0)$. As $(2,p-1)=(2,q-1)$ all diagonal automorphisms are induced by elements of $\tilde{T}^{F_0}$. Therefore, we can assume by possibly replacing $\chi$ by its $\tilde{G}$-conjugate character that $\chi$ is a constituent of $\Gamma$.
	
 Let $\hat{\phi}_0 \in \Irr(U \langle F_0 \rangle)$ be the trivial extension of $\phi_0$ and set $\hat{\Gamma}:=\Ind_{U \langle F_0 \rangle}^{G \langle F_0 \rangle}(\hat{\phi}_0)$. By Mackey's formula $\Res_{G}^{G \langle F_0 \rangle }(\hat{\Gamma})=\Gamma$ and $\Gamma$ is multiplicity free. By Frobenius reciprocity, there exists a constituent $\hat{\chi}$ of $\hat{\Gamma}$ of multiplicity one, which is an extension of $\chi$. The group $\tilde{G}$ stabilizes $\mathcal{E}(G,s)$ and acts transitively on it. Hence, $\tilde{G} (\tilde{G} \mathcal{G})_\chi=\tilde{G} \mathcal{G}_{\mathcal{E}(G,s)}$. Let $\tau \in \gal$ be a Galois automorphism stabilizing the set $\mathcal{E}(G,s)$. Then we find $\tilde{t} \in \tilde{T}^{F_0}$ such that $\phi_0^{\tilde{t} \tau}= \phi_0$, see e.g. \cite[Lemma 6.3]{ruhstorfer} or by direct calculation. Thus, $\chi^{\tilde{t} \tau} \in \mathcal{E}(G,s)$ and $\chi \in \mathcal{E}(G,s)$ are both constituents of $\Gamma$. From this we deduce that $\chi^{\tilde{t} \tau}=\chi$. Moreover, as $\tilde{t}$ is $F_0$-stable the character $\hat{\phi}_0$, and therefore also $\hat{\Gamma}$, is $\tilde{t} \tau$-stable. It follows that $\hat{\chi}$ is $\tilde{t} \tau$-stable as well. Thus,  $\tilde{G} (\tilde{G} \mathcal{G})_{\hat\chi}=\tilde{G} \mathcal{G}_{\mathcal{E}(G,s)}$.
	}
\end{proof}

Now, letting $P\in\Syl_2(G)$, there is an appropriate $\aut(G)_P$-stable subgroup $M$ satisfying $\norm_{G}(P)\leq M<G$, defined as in \cite[Theorem 15.3]{IsaacsMalleNavarroMcKayreduction} if $n=1$, \cite[Theorem 7.8]{malle07} if $q\equiv 1\pmod 8$ with $n\geq 2$, and as in  \cite[Section 4.4]{malle08} otherwise. Throughout the next proof, we let $M$ be this group.

\begin{proof}[Proof of \prettyref{thm:iMN} for $S=\PSp_{2n}(q)$]
Keep the notation from before. Let $P\in\syl_2(G)$ and write $\galh:=\galh_2$.  Then for $n \geq 2$, by \cite[Theorem 6.2 and Proposition 7.3]{SF20}, there is an $(\aut(G)_P\times \galh)$-equivariant bijection 
\[\Omega\colon \irr_{2'}(G)\rightarrow \irr_{2'}(M).\]  

When $n=1$, note that the assumption $S$ is simple implies $q\geq 5$.  Further, the case $S=\PSL_2(9)\cong \alt_6\cong \type{B}_2(2)'$ has been handled in \cite[Proposition 5.3]{johansson}.  Hence we may assume $G$ is the universal covering group for $S$. The odd-degree characters of $G$ are the two unipotent characters $1_G, \mathrm{St}_G,$ and the two characters of degree $\frac{1}{2}(q+\epsilon)$, where $q\equiv\epsilon\mod 4$ with $\epsilon\in\{\pm1\}$. The two unipotent characters are rational-valued, and the remaining two odd-degree characters are stabilized by $\galh$ when $q\equiv \pm1\pmod 8$ and are fixed by $\tau\in\galh$ if and only if $\tau$ fixes $\sqrt{\epsilon p}$ when $q\equiv \pm3\pmod 8$.   When $q\equiv \pm1\pmod 8$, the group $M$ is the normalizer $N_0$ of a Sylow $d_2(q)$-torus, which is of the form $C_{q-\epsilon}.2$ with the $.2$ acting via inversion and induced by an element $c$ such that $c^2$ has order $2$. Here $C_{q-\epsilon}.2/\langle c^2\rangle$ is a semidirect product $C_{q-\epsilon}\rtimes C_2$.
 In this case, the odd-degree characters of $M$ are the extensions to $M$ of the trivial character and the unique character of order $2$ of $C_{q-\epsilon}$, both of which are trivial on $c^2$.  It follows that the members of $\irr_{2'}(M)$ here are then  rational-valued.  When $q\equiv \pm3\pmod 8$, the group $M$ is isomorphic to $\SL_2(3)$, so the set $\irr_{2'}(M)$ is analogous to $\irr_{2'}(G)$, comprised of $\{1_M, \mathrm{St}_M\}$ and two characters of degree $1$ whose behavior under $\galh$ is the same as the non-unipotent characters in $\irr_{2'}(G)$.   From this, we can see that the $\aut(G)_P$-equivariant bijection $\Omega\colon\irr_{2'}(G)\rightarrow\irr_{2'}(M)$ in \cite[Theorem 15.3]{IsaacsMalleNavarroMcKayreduction} is also $\galh$-equivariant. 

 In all cases, it therefore suffices to show that 
\[(G\rtimes \aut(G)_{P, \chi^{\galh}}, G, \chi)_{\galh} \geqslant_c (M\rtimes \aut(G)_{P, \chi^{\galh}}, M, \Omega(\chi))_{\galh}\] for all $\chi\in\irr_{2'}(G)$, in the notation of \cite[Definition 1.5]{NSV}.  Let $\wt{M}:=M \norm_{\wt{G}}(P)$.  Since $\wt{G} D$ induces all automorphisms of $G$, and $M$ can be chosen to be $D$-stable, \cite[Theorem 2.9]{NSV} further implies that it suffices to show
\[((\wt{G} D)_{ \chi^{\galh}}, G, \chi)_{\galh} \geqslant_c ((\wt{M}D)_{ \chi^{\galh}}, M, \Omega(\chi))_{\galh}\] for all $\chi\in\irr_{2'}(G)$.  
We remark that the groups involved satisfy part (i) of \cite[Definition 1.5]{NSV} and that since $\Omega$ is $\wt{M}D\times\galh$-equivariant, part (ii) is also satisfied.

	
	Let $\chi\in\irr_{2'}(G)$. By \cite[Theorem 2.5 and Proposition 4.9]{malle08}, we have $(\wt{G} D)_\chi=\wt{G} D$ if $\chi$ is a unipotent character and $(\wt{G}  D)_{ \chi} ={G} Z(\wt{G}) \rtimes D_\chi$ otherwise.

 Note that since $Z(G)$ is a $2$-group, $\chi$ is trivial on $Z(G)$ (since $\chi$ has odd degree and $G$ is perfect), so through deflation and inflation, we may view $\chi$ as a character of $G Z(\wt{G})$ trivial on $Z(\wt{G})$.  Throughout, we will sometimes make this identification. In particular, since the maps $\Omega$ are constructed so that $\chi$ and $\Omega(\chi)$ lie over the same character of $Z(G)$, their extensions to $(\wt{G}D)_\chi$ and $(\wt{M}D)_\chi$, which exist by \cite[Proposition 4.9]{malle08} (see also the proofs of \cite[Theorems 4.10 and 4.11]{malle08}), 
 may be taken to lie over the same (namely, trivial) character of $Z(\wt{G})$, giving part (iii) of \cite[Definition 1.5]{NSV}.  Hence, it suffices to find extensions of $\chi$ and $\Omega(\chi)$ to   $(\wt{G} D)_\chi$ and $(\wt{M} D)_\chi$, respectively, satisfying part (iv) of \cite[Definition 1.5]{NSV}.  We also note that by \cite[Proposition 4.5]{malle08}, $\chi$ lies in a rational Lusztig series $\mathcal{E}(G, s)$ with $s^2=1$. 

\medskip

(0) Let $n=1$, so $G=\SL_2(q)$.   Let $\tau\in \galh$ and let $\chi\in\irr_{2'}(G)$.  Note that $\chi$ is either $\tau$-invariant or $\chi^\tau=\chi^{\tilde{t}}$ for $\tilde{t}$ inducing the nontrivial diagonal automorphism in $\out(G)$.  First suppose that $q\equiv \pm1\pmod 8$.  Then in the proof of Lemma \ref{lem:SL2}, we may take $\wt{t}=1$, since all members of $\irr_{2'}(G)$ are $\galh$-invariant (see the discussion in the second paragraph of the proof). Hence, there is a $\tau$-invariant extension $\hat\chi$ to $G\langle F_0\rangle$. Recalling that the odd-degree characters of $M$ are linear and rational-valued here, we may extend them trivially to $M\langle F_0\rangle$ as well, to obtain the desired extensions when $\chi$ is not unipotent.  When $\chi\in\{1_G, \mathrm{St}_G\},$ there is a rational extension to $\wt{G}D$. Here $\Omega(\chi)$ is trivial on $C_{q-\epsilon}$ and we have $\wt{M}=\wt{C}.2$, where $\wt{C}$ is an abelian group containing $C_{q-\epsilon}$.  Then taking the two extensions of the trivial character of $\wt{C}$, we may argue as before to get extensions of $\Omega(\chi)$ to $\wt{M}\langle F_0\rangle$.  

Now let $q\equiv \pm3\pmod 8$.  Then since $M\cong \SL_2(3)$, the required extensions exists if $\chi$ (and hence $\Omega(\chi)$) are unipotent, as before.  So assume that $\chi$ is one of the non-unipotent members of $\irr_{2'}(G)$.  Since $F_0$ centralizes $M$, $\Omega(\chi)$ extends to a character of $M\langle F_0\rangle$ with the same field of values as $\Omega(\chi)$.  Let $\tilde{t}$ be as in the proof of Lemma \ref{lem:SL2}.  Recall that $\chi$ is $\tau$-invariant if and only if $\Omega(\chi)$ is, in which case we may take $\tilde{t}=1$.  If $\chi^\tau\neq \chi$, we have $\chi^{\tilde{t}\tau}=\chi$ with $\tilde{t}$ inducing the nontrivial diagonal automorphism, and the same is true for $\Omega(\chi)$.  (Note that such a $\tilde{t}$ also induces the nontrivial diagonal automorphism on $M\cong \SL_2(3)$.)  By Lemma \ref{lem:SL2}, we have an extension $\hat\chi$ to $G\langle F_0\rangle$ that also satisfies $\hat\chi^{\tilde{t}\tau}=\hat\chi$.  Since the same holds for $\Omega(\chi)$, we have obtained the desired extensions.  

\medskip

(1) Now, suppose that $n\geq 2$ and $q\equiv 1\pmod 8$, and let $\chi\in\irr_{2'}(G)$.  Then by \cite[Lemma 4.10 and Theorem B]{SFT20}, $\Q(\chi)=\Q$ and hence $\chi^{\galh}=\chi$, so we aim to show  
\[((\wt{G} D)_{ \chi}, G, \chi)_{\galh} \geqslant_c ((\wt{M}D)_{ \chi}, M, \Omega(\chi))_{\galh}\]
in this case.  Further, by \cite[Lemma 7.5 and Theorem 7.7]{MS16}, $\chi$ lies in a principal series $\mathcal{E}(G, (T,\delta))$ with $\delta^2=1$.    Note that $\delta$ is $D\times \galh$-invariant, so $\galh=\galh_{\chi}=\galh_{\delta}$.   In this case, the group $M$ is $\norm_{G}(\bT)$, where $\bT$ is a maximally split torus in $\bG=\Sp_{2n}(\overline{\F}_q)$, and the map $\Omega$  is induced by one as in  \prettyref{lem:principalextend}(2).   

Now, \prettyref{lem:principalextend} yields an extension $\hat\chi$ of $\chi$ to $G\langle F_0\rangle$ and $\hat\psi$ of $\Omega(\chi)$ to $M\langle F_0\rangle$ that are each $\galh$-invariant.  If $\chi$ is not unipotent, this implies that there are $\galh$-invariant extensions to $(\wt{G} D)_\chi=G Z(\wt{G})\rtimes D_\chi$ and $(\wt{M}D)_\chi=MZ(\wt{G})\rtimes D_{\chi}$, giving part (iv) of \cite[Definition 1.5]{NSV} in this case.

  Now assume $\chi$ is unipotent.  Since $\chi$ is unipotent and lies in a principal series, we have $\chi\in \mathcal{E}(G, (T, 1))$ and extends to a unipotent character $\wt{\chi}$ in the principal series $\mathcal{E}(\wt{G}, (\wt{T}, 1))$ for $\wt{G}$.  The character $\wt{\chi}$ is also rational-valued and  satisfies $D_{\wt{\chi}}=D_\chi$.  On the other hand, $\Omega(\chi)$ may be identified with an odd-degree character of $W=M/T\cong \wt{M}/\wt{T}$.  Then the same arguments as in \prettyref{lem:principalextend} yields a rational-valued character of $\wt{G}\langle F_0\rangle$ extending $\wt{\chi}$, and hence $\chi$, as well as a rational-valued character of $\wt{M}\langle F_0\rangle$ extending $\Omega(\chi)$, again giving part (iv) of \cite[Definition 1.5]{NSV}. 
  
  \medskip
  
We may therefore assume that $n\geq 2$ and  $q$ is an odd power of an odd prime.

\medskip

(2) Assume $n=2^k\geq 2$ is a power of $2$. Here $M$ is of the form $\Sp_{n}(q)\wr C_2$ and the non-unipotent members of $\irr_{2'}(X)$ for $X\in\{G, \Sp_n(q)\}$ lie in a single principal series $\mathcal{E}(X, (T_X, \delta_X))$ with $T_X$ a maximally split torus of $X$ and $\delta_X\in\irr(T_X)$ with $\delta_X^2=1$. The map $\Omega$ sends unipotent characters to the extensions in $M$ of characters $\mu\otimes\mu\in\irr(\Sp_n(q)\times \Sp_n(q))$ for unipotent $\mu$ and sends non-unipotent characters to corresponding extensions for non-unipotent $\mu$.  (See \cite[Lemma 7.2]{SF20} for details.)

The same arguments as above yield extensions of $\chi$ to $\wt{G}_\chi\langle F_0\rangle$, and of $\mu$ to $\operatorname{CSp}_n(q)_\mu\langle F_0\rangle$ that are invariant under any $\tau\in\galh_{\chi}$.  Here since $F_0$ induces the action of $F_p$ on each $\Sp_n(q)$ component, we by an abuse of notation write $F_0$ for the corresponding field automorphism of $\Sp_n(q)$. Further, note that the action of $C_2$ on $\Sp_n(q)\times \Sp_n(q)$ (viewed as being embedded into $G$ block-diagonally) is induced by the  matrix $\begin{pmatrix} 0 & I_n \\ I_n & 0\end{pmatrix}$, so we see that $F_0$ commutes with this element. With this, we have an extension of $\Omega(\chi)$ to $\wt{M}_\chi\langle F_0\rangle$ invariant under such a $\tau$, applying \cite[Lemma 5.5]{SF20}.  If $q\equiv 7\pmod 8$, $\galh_\chi=\galh$ by \cite[Lemma 4.10 and Theorem B]{SFT20} and this completes part   (iv) of \cite[Definition 1.5]{NSV}.

So, suppose $q\equiv \pm 3\pmod 8$ and $\tau\in\galh$ does not fix $\chi$.  Then $\chi$ is non-unipotent and we have  $\chi\in\mathcal{E}(G, (T, \delta))$ with $\delta^2=1$, as above. 
Suppose that $a\in \wt{M}D$ satisfies $\chi^{a\tau}=\chi$.  Then since $\chi$ is $MD$-invariant, we may without loss assume that $a\in \wt{T}$ induces the nontrivial diagonal automorphism of order 2 in $\out(G)$ and that $[a, F_0]=1$. 
Here $\wt{T}$ is a maximally split torus of $\wt{G}$ containing $T$.   Note then that $a$ preserves $\bT$ and $\delta^{a\tau}=\delta$.   

As in \prettyref{lem:principalextend}, extending $\delta$ to $\wh{\delta}$ in $\irr(T\langle F_0\rangle)$ via $\wh{\delta}(F_0)=1$ yields $\wh{\delta}^{a\tau}=\wh{\delta}$. Let $\hat\chi\in\irr(G\langle F_0\rangle|\chi)$ be such that $\langle \HC_{T\langle F_0\rangle}^{G\langle F_0\rangle}\hat\delta, \hat\chi\rangle\neq 0$ as in \prettyref{prop:uniquedescent}.  Note that $\langle \HC_{T\langle F_0\rangle}^{G\langle F_0\rangle}\hat\delta, \hat\chi^{a\tau}\rangle\neq 0$ and that $\hat\chi^{a\tau}$ lies above $\chi$.  This forces $\hat\chi^{a\tau}=\hat\chi$ by \prettyref{prop:uniquedescent}.  This then implies that there is an $a\tau$-invariant extension of $\chi$ to $(\wt{G}D)_\chi$.

Arguing analogously for the characters of $\Sp_n(q)$ and noting that $a$ induces the corresponding diagonal automorphism on the two $\Sp_n(q)$ components, we obtain an $a\tau$-invariant extension of $\mu\otimes \mu$ to $(\Sp_n(q)\times \Sp_n(q))\langle F_0\rangle$, and hence an $a\tau$-invariant extension of $\Omega(\chi)$ to $M\langle F_0\rangle$, and hence $(\wt{M}D)_\chi$.  (Indeed, since $a$ may be chosen to commute with the $C_2$-action, arguing exactly as in \cite[Lemma 5.5]{SF20} for $a\tau$ yields such an extension.) This completes part (iv) of \cite[Definition 1.5]{NSV} in this case.

 \medskip

(3) Now assume the $2$-adic decomposition of $n$ is $n=\sum_{j\in J} 2^j$ with $|J|\geq 2$, and write $k:=\mathrm{max}\{j\in J\}$ and $m:=2^k$. Then in this case, $M=\Sp_{2(n-m)}(q)\times \Sp_{2m}(q)$. Write $M_1:=\Sp_{2(n-m)}(q)$ and $M_2:= \Sp_{2m}(q)$.  The map $\Omega$ sends unipotent characters in $\irr_{2'}(G)$ to products of unipotent characters in $\irr_{2'}(M)$. Again, every odd-degree unipotent character of $\irr_{2'}(X)$ for $X\in\{G, M_1, M_2\}$ lies in the principal series, and hence arguing as before yields part (iv) of \cite[Definition 1.5]{NSV} if $\chi$ is unipotent.

We therefore assume $\chi$ is not unipotent, so $\Omega(\chi)=\chi_1\otimes\chi_2\in\irr_{2'}(M_1)\otimes \irr_{2'}( M_2)$ with at least one of $\chi_1$ or $\chi_2$ non-unipotent. Note that the diagonal and field automorphisms of $G$ induce the corresponding automorphism on $M_1$ and $M_2$.  Also, recall that non-unipotent members of $\irr_{2'}(M_2)$ lie in a single principal series as in part (2). Further, if both $\chi$ and $\chi_1$ also lie in a principal series, then arguing as in part (2) gives the desired extensions for $\chi, \chi_1,$ and $\chi_2$, and gives part (iv) of \cite[Definition 1.5]{NSV}.

Now assume $\chi$ does not lie in a principal series.  Then $q\equiv 3\pmod 4$, $n$ is odd, and $\chi$ lies in a Harish-Chandra series $\mathcal{E}(G, (L, \delta))$ where $L\cong \Sp_2(q)\times T_0$ with $T_0=\bT_0^F$ a maximally split torus of $\Sp_{2n-2}(q)$ and $\delta=\psi\otimes \delta_0$ with $\delta_0\in\irr(T_0)$ satisfying $\delta_0^2=1$ and $\psi$ one of the two characters of $\Sp_2(q)$ of degree $\frac{q-1}{2}$.  (See \cite[Theorem 7.7]{MS16}.)  Analogous odd-degree characters exist for $M_1$, and the map $\Omega$ is constructed so that $\chi_1$ is also of this form. Say $\chi_1\in\mathcal{E}(M_1, (L_1, \delta_1))$ with $\delta_1:=\psi\otimes \delta_{0,1}\in \irr(\Sp_2(q))\otimes \irr(T_{0,1})$.  (Here $\psi$ is the same as for $\chi$.)

Now, \cite[Theorem B]{SFT20}, together with the explicit knowledge of the values of $\psi$, yields that $\chi$ is fixed by $\tau\in\galh$ if and only if $\psi$ is, and that both are fixed by $\tau^2$.  Further, recall that \cite[Theorem 4.9]{malle08} yields that $\chi$ and $\psi$ have the same stabilizer in $\out(G)$. Note that $F_0$ induces the corresponding field automorphism on the $\Sp_2(q)$ and $T_0$ components of $L$  and that $\norm_G(\bL)=\Sp_2(q)\times \norm_{\Sp_{2n-2}(q)}(\bT_0)$.  Then since $\psi$ is $F_0$-invariant, we see $\delta$ satisfies the assumptions of Proposition \ref{prop:uniquedescent} since $\delta_0$ does in $\Sp_{2n-2}(q)$, being of the same form as the case of principal series characters in $\Sp_{2n-2}(q)$.  

 Assume that $\chi^{a\tau}=\chi$ for $a\in \wt{M}D$ and $\tau\in\galh$.  If $a$ is nontrivial, we may again assume that $a\in\wt{T}$ induces the nontrivial diagonal automorphism and that it further induces 
 a non-inner diagonal automorphism on $\Sp_2(q)$ and $T_0$. As in part (2), $\delta_0$ is invariant under $\tau, F_0, $ and $a$.  


As in part (2), $\delta_0$  may be extended to a character $\wh{\delta_0}$ of $T_0\langle F_0\rangle$ satisfying $\wh{\delta_0}^{a\tau}=\wh{\delta_0}$. Let $\wh{\psi}$ be some extension of $\psi$ to a character of $\Sp_2(q)\langle F_0\rangle$ with $\wh{\psi}^{a\tau}=\wh{\psi}\la$ for $\la\in\irr(\langle F_0\rangle)$.  Then the extension $\wh{\delta}:=(\wh{\psi}\otimes \wh{\delta}_0)_{L\langle F_0\rangle}$ of $\delta$ to $L\langle F_0\rangle$ satisfies $\wh{\delta}^{a\tau}=\wh{\delta}\la$.  Let $\hat\chi\in\irr(G\langle F_0\rangle|\chi)$ be such that $\langle \HC_{L\langle F_0\rangle}^{G\langle F_0\rangle}\hat\delta, \hat\chi\rangle\neq 0$ as in \prettyref{prop:uniquedescent}.  Note that $0\neq \langle \HC_{L\langle F_0\rangle}^{G\langle F_0\rangle}\hat\delta^{a\tau}, \hat\chi^{a\tau}\rangle=\langle \HC_{L\langle F_0\rangle}^{G\langle F_0\rangle}\hat\delta \la, \hat\chi^{a\tau}\rangle=\langle \HC_{L\langle F_0\rangle}^{G\langle F_0\rangle}\hat\delta, \hat\chi^{a\tau}\la^{-1}\rangle$ and that $\hat\chi^{a\tau}\la^{-1}$ lies above $\chi$.  This forces $\hat\chi^{a\tau}\la^{-1}=\hat\chi$ by \prettyref{prop:uniquedescent}, and hence $\wh{\chi}^{a\tau}=\wh{\chi}\la$.

Now, since $\Omega$ is $\wt{M}D$-equivariant, we have $\chi_1^{a\tau}=\chi_1$ and $\chi_2^{a\tau}=\chi_2$.  Recalling that $\delta_1$ is defined using the same $\psi$ as for $\delta$, the exact same arguments can be made for $\delta_1$ and $\chi_1$, so there is an extension $\wh{\chi}_1$ of $\chi_1$ to $M_1\langle F_0\rangle$ such that $\wh{\chi}_1^{a\tau}=\wh{\chi}_1\lambda$.  The argument in part (2) yields an extension $\wh{\chi}_2$ of $\chi_2$ to $M_2\langle F_0\rangle$ such that $\wh{\chi}_2^{a\tau}=\wh{\chi}_2$.  Together, the extension $\wh{\Omega(\chi)}:=(\wh{\chi}_1\otimes \wh{\chi}_2)_{M\langle F_0\rangle}$ of $\Omega(\chi)$ to $M\langle F_0\rangle$ satisfies $(\wh{\Omega(\chi)})^{a\tau}=\wh{\Omega(\chi)}\la$, giving part (iv) of \cite[Definition 1.5]{NSV} in this final case and completing the proof.
\end{proof}

%

\section{The Proof of Theorem \ref{thm:iMN}: Remaining Cases}\label{sec:remaining}

We keep the notation of Section \ref{sec:notn}.  We now assume that $S=G/Z(G)$ is one of the simple groups listed in Theorem \ref{thm:iMN}. Then $\wt{G}D$, where $D$ is a well-chosen group of field and graph automorphisms, induces all automorphisms of $S$.

It will also be useful to keep the notation of \cite[Section 2.2]{SF20}, so that  $v\in\bG$ is the canonical representative in the extended Weyl group of $G$ of the longest element in the Weyl group $\mathbf{W}:=\norm_\bG(\bT)/\bT$, which induces an isomorphism $G\cong \bG^{vF}$; $\bT=\norm_\bG(\bS)$ for a Sylow $d$-torus $\bS$ of $(\bG, vF)$; $T_1:=\bT^{vF}$; and $N_1:=\norm_{\bG^{vF}}(\bS)$.  Further let $\wt{N}_1:=\norm_{\wt{\bG}^{vF}}(\bS)$. (We remark that when $d=1$, we have $T_1=T$, $N_1=N$, and $\wt{N}_1=\norm_{\wt{G}}(\bT)$.)

%
%

\begin{proof}[Proof of Theorem \ref{thm:iMN}]
From \prettyref{sec:typeC}, we may assume that $S\neq \PSp_{2n}(q)$ is one of the remaining groups in \prettyref{thm:iMN}. Further, recall that we assume that $G$ is the universal covering group of $S$. (That is, $S\neq \type{B}_3(3)$.) Let $P\in\syl_2(G)$,   $\galh:=\galh_2$, and let $d\in\{1,2\}$ be the order of $q$ modulo 4.

By  \cite[Theorem 6.2 and Proposition 8.1]{SF20}, there is an $\aut(G)_P\times\galh$-equivariant bijection \[\Omega\colon \irr_{2'}(G)\rightarrow \irr_{2'}(M),\] where $M:=\norm_G(\bS_0)$ for a Sylow $d$-torus $\bS_0$ of $(\bG, F)$.     Let $\chi\in\irr_{2'}(G)$.  By \cite[Proposition 8.1]{SF20}, $\chi$ is rational-valued, so that $\chi^\galh=\chi$.  By  \cite[Lemma 7.5 and Theorem 7.7]{MS16} and \cite[Lemma 2.6]{SF20}, $\chi$ lies in a principal series $\mathcal{E}(G, (T,\delta))$ with $\delta^2=1$. Note also that $\Omega$ is as in Lemma \ref{lem:principalextend}(2) if $q\equiv 1\pmod 4$.

Write $\psi:=\Omega(\chi)$ and $\wt{M}:=\norm_{\wt{G}}(\bS_0)$.  Note that, as in the proof for symplectic groups, it again suffices to show
\[((\wt{G} D)_{ \chi}, G, \chi)_{\galh} \geqslant_c ((\wt{M}D)_{ \chi}, M, \psi)_{\galh}\] for all $\chi\in\irr_{2'}(G)$, and that we again have part (i) and (ii) of \cite[Definition 1.5]{NSV} are satisfied. Note that $D=\langle F_p\rangle$ for a generating field automorphism $F_p$. When $q\equiv 3\pmod 4$, by applying \cite[Lemma 2.1]{NSV}, it will be useful to identify $M$ with $N_1$, $G$ with $\bG^{vF}$, and $\wt{M}$ with $\wt{N}_1$, in which case there is a field automorphism $\hat{F}_p$ that acts on $\bG^{vF}$ as $F_p$ and so that $T_1$ and $N_1$ are stabilized by $\hat{F}_p$ (see  Notation 3.3 and the proof of Proposition 3.4 of \cite{MS16}.)

Now, note that $|\wt{G}/G\zen(\wt{G})|$ divides $2$ and that $D_\chi=\langle F_0\rangle=D_\psi$ for some field automorphism $F_0$.  By \cite[Theorem 2.5.12(g)]{gorensteinlyonssolomonIII}, we have $F_0$ commutes with the action of $\wt{G}$ up to conjugation by $G$. Write $\hat{G}=G\langle F_0\rangle$ and $\hat{M}=M\langle F_0\rangle$.  

 By \cite[Corollary 5.3]{MS16}, we have $(\wt{G} D)_{\chi_0}=\wt{G}_{\chi_0} D_{\chi_0}$ for some $\wt{G}$-conjugate $\chi_0$ of $\chi$. However, this gives $(\wt{G} D)_\chi=\wt{G}_\chi D_\chi$ since $D$ and $\wt{G}$ commute up to the action of $G$.  We further have $(\wt{M}D)_\psi=\wt{M}_\psi D_\psi$, arguing similarly and using \cite[Theorems 3.1 and 3.18]{MS16}.  We have part (iii) of \cite[Definition 1.5]{NSV} by \cite[Lemma 2.11]{spath12} and \cite[Lemma 2.13]{spath12} and its proof.  We are therefore left to prove that part (iv) of \cite[Definition 1.5]{NSV} holds.

Let $a\tau\in \wt{M}D\times \galh$ such that $\chi^{a\tau}=\chi$.  Then $a\in (\wt{M}D)_\chi$ since $\chi$ is rational-valued.
By Lemma \ref{compute mu},  
 it suffices to show that there are extensions $\wt{\chi}$ and $\hat\chi$, respectively $\wt{\psi}$ and $\hat\psi$, of $\chi$ to $\wt{G}_\chi$ and $\hat G$, resp. of $\psi$ to $\wt{M}_\psi$ and $\hat M$, such that  if  $\wt{\chi}^{a\tau}=\mu_1\wt{\chi}$, $\wt{\psi}^{a\tau}=\mu_1'\wt{\psi}$, $\wh{\chi}^{a\tau}=\mu_2\wh{\chi}$, $\wh{\psi}^{a\tau}=\mu_2'\wh{\psi}$, then $\mu_i=\mu_i'$ for $i=1,2$.    
(Throughout, we will identify the characters of $\wt{T}/T$, $\wt{G}/G$, and $\wt{M}/M$, and similar for $\hat{T}/T$, $\hat G/G$ and $\hat M/M$.)

\medskip

We first show that there are appropriate extensions $\hat\chi$ and $\hat\psi$ as above such that $\mu_2=\mu_2'$.  We may assume that $a\in\wt{M}_\chi$ since certainly extensions $\hat\chi$ and $\hat\psi$ to $\hat G$ and $\hat M$ will be $F_0$-invariant.  Hence we may assume $a\in\wt{T}_\chi$.  Note that $a$ preserves $\hat T$.  
If $q\equiv 1\pmod 4$, then Lemma \ref{lem:principalextend} yields extensions $\hat\chi$ and $\hat\psi$ such that $\hat\chi^\tau=\hat\chi$ and $\hat\psi^\tau=\hat\psi$.  Let $\mu_2\in\irr(\hat G/G)$ such that $\hat\chi^a=\mu_2\hat\chi$.  In the notation of the proof of Lemma \ref{lem:principalextend}, we have $\langle \HC_{\hat T}^{\hat G} \hat\delta^a, \hat\chi^a\rangle\neq 0$.  This forces $\hat\delta^a=\mu_2\hat\delta$ by Proposition \ref{prop:uniquedescent}, since $\HC_{\hat T}^{\hat G} (\mu_2\hat\delta)=\mu_2\HC_{\hat T}^{\hat G} \hat\delta$ by \cite[Problem (5.3)]{isaacs}.  
Now, recall that $\psi=\Ind_{ N_\delta}^{ N} (\Lambda(\delta)\gamma)$ and $\hat\psi=\Ind_{\hat N_\delta}^{\hat N} (\hat\Lambda(\hat\delta)\gamma)$, 
where $\Lambda$ is the extension map with respect to $T\lhd N$ studied in \cite[Section 4]{SF20}, $\gamma$ is a linear character of $N_\delta/T\cong \hat N_\delta/\hat T$, and $\hat\Lambda(\hat\delta)$ is the unique common extension of $\Lambda(\delta)$ and $\hat\delta$ to $\hat N_\delta$.  Let $\nu\in N_\delta/T$ be such that $\Lambda(\delta)^a=\Lambda(\delta)\nu$.  Since $\psi^a=\psi$, the uniqueness of Clifford correspondence and \cite[Proposition 3.15]{MS16} yields that $\Lambda(\delta)\nu\gamma^a=\Lambda(\delta)\nu\gamma=\Lambda(\delta)\gamma$.   That is, $\Lambda(\delta)^a=\Lambda(\delta)$ and $\gamma^a=\gamma$.  Then it follows that $\hat\Lambda(\hat\delta)^a=\mu_2\hat\Lambda(\hat\delta)$ and $\hat\psi^a=\Ind_{\hat N_\delta}^{\hat N} (\mu_2\hat\Lambda(\hat\delta)\gamma)=\mu_2\hat\psi$, completing the claim in this case.

If $q\equiv 3\pmod 4$, note that $q$ is not a square, so $F_0$ has odd order.  Here $\irr_{2'}(N_1)$ is in bijection with pairs $(\delta_1, \eta_1)$ where $\delta_1\in\irr(T_1)$ satisfies $[N_1:(N_1)_{\delta_1}]$ is odd,  $\delta_1^2=1$, and $\eta_1\in\irr_{2'}(W_1(\delta_1))$.  Here we define $W_1(\delta_1):=(N_1)_{\delta_1}/T_1$.
Now, the member of $\irr_{2'}(N_1)$ corresponding to $(\delta_1, \eta_1)$ is of the form $\mathrm{Ind}_{(N_1)_{\delta_1}}^{N_1}(\Lambda_1(\delta_1)\eta_1)$, where $\Lambda_1$ is an extension map with respect to $T_1\lhd N_1$.  (See \cite[Section 8]{SF20}.)  
Then in this case, since the order of $F_0$ is relatively prime to $o(\delta)$ and $o(\delta_1)$, we may 
appeal to \prettyref{lem:principalextend}(1) and \prettyref{lem:localextend} with $(X,Y, \alpha, \beta)=(T_1, N_1, \hat{F}_0, a)$, to obtain the extensions such that $\mu_2=\mu_2'$.

\medskip

Finally, it remains to show that we may obtain extensions $\wt{\chi}$ and $\wt{\psi}$ as described above such that $\mu_1=\mu_1'$ when $\wt{G}_\chi=\wt{G}$ in the cases of $\type{B}_n(q)$ or $\type{E}_7(q)$.   It suffices here to assume that $a\in D_\chi=\langle F_0\rangle$, since certainly an extension of $\chi$ or $\psi$ to $\wt{G}$ or $\wt{M}$ is fixed by $\wt{M}$. Further, note that $|\zen(G)|=2$, so $\chi\in\irr_{2'}(G)$ is necessarily trivial on the center.  Hence we may trivially extend $\chi$ to $\chi'\in\irr_{2'}(G\zen(\wt{G}))$, and we may similarly extend $\psi$ to $\psi'\in\irr_{2'}(M\zen(\wt{G}))$. Then since $|\wt{G}/G\zen(\wt{G})|=2$, we wish to show that $\wt{\chi}\in\irr(\wt{G}|\chi')$ is fixed by $a\tau$ if and only if $\wt{\psi}\in\irr(\wt{M}|\psi')$ is.  Now, from \cite[Theorem 6.3]{MS16}, we see there is a $D_\chi$-equivariant map $\irr(\wt{G}|\chi')\rightarrow \irr(\wt{M}|\psi')$, so it suffices to show that $\wt{\chi}$ is fixed by $\tau$ if and only if $\wt{\psi}$ is. 

Now, recall that $\chi$ lies in a principal series $\mathcal{E}(G, (T, \delta))$ with $\delta^2=1$.  Let $\wt{T}=\wt{\bT}^F$, where $\wt{\bT}:=\bT\zen(\wt{\bG})$ is a maximally split torus of $\wt{\bG}$ containing $\bT$.  Then an application of Mackey's theorem yields that $\wt{\chi}\in\mathcal{E}(\wt{G}, (\wt{T}, \wt{\delta}))$ for some $\wt{\delta}\in\irr(\wt{T}|\delta)$.    Since the corresponding relative Weyl group $\norm_{\wt{G}}(\wt{\bT})_{\wt{\delta}}/\wt{T}$ is a Weyl group, \cite[Theorem 3.8]{SFgaloisHC}  (see also \cite[Proposition 5.5]{geck03}) implies that $\wt{\chi}^\tau=\wt{\chi}$ if and only if $\wt{\delta}^\tau=\wt{\delta}$. 
 
 On the other hand, recall that $\psi=\mathrm{Ind}_{(N_1)_{\delta_1}}^{N_1}(\Lambda_1(\delta_1)\eta_1)$.  Note that 
 \begin{equation}\label{eq:Lambda_1}\Lambda_1(\delta_1)^\tau=\Lambda_1(\delta_1)=\Lambda_1(\delta_1^\tau)
 \end{equation} by \cite[Proposition 4.6 and the last paragraph of Proposition 8.1]{SF20}.
 Let $\wt{\delta}_1$ be an extension of $\delta_1$ to $\wt{\bT}^{vF}$ and let $\wt{\Lambda}(\wt{\delta}_1)\in\irr((\wt{N}_1)_{\wt{\delta}_1})$
  be the unique common extension of $\Lambda_1(\delta_1)|_{(N_1)_{\wt{\delta}_1}}$ and $\wt{\delta}_1$ as in the proof of \cite[Proposition 3.20]{MS16}. Similarly, let $\wt{\Lambda}(\wt{\delta}_1^\tau)$ be the unique common extension of $\Lambda_1(\delta_1)|_{(N_1)_{\wt{\delta}_1}}$ and $\wt{\delta}_1^\tau$.  Then using \eqref{eq:Lambda_1} and by uniqueness, this forces $\wt{\Lambda}(\wt{\delta}_1^\tau)=\wt{\Lambda}(\wt{\delta}_1)^\tau$.
  
  Now, \cite[Problems (5.1)-(5.3)]{isaacs} imply that $\wt{\psi}$ is of the form $\mathrm{Ind}_{(\wt{N}_1)_{\wt{\delta}_1}}^{\wt{N}_1}(\wt{\Lambda}(\wt{\delta}_1)\wt{\eta}_1)$ for some linear $\wt{\eta}_1\in\irr((\wt{N}_1)_{\wt{\delta}_1}/\wt{T}_1)$.    Note that $\wt{\eta}_1^\tau=\wt{\eta}_1$ since it is a linear character of a real reflection group. 
   Hence $\wt{\psi}^\tau=\wt{\psi}$ if and only if $\wt{\delta}_1^\tau=\wt{\delta}_1$.  In the case $q\equiv 1\pmod 4$ (and hence $N=N_1=M$), this completes the proof.  Now, note that $\wt{\delta}_1$ is in duality with $(\wt{\bT}^{vF}, \wt{s_1})$, where $\wt{\chi}$ lies in the Lusztig series of $\wt{s_1}$ when viewed as a character of $\wt{\bG}^{vF}$.  But we also have $\wt{\delta}$ is in duality with $(T, \wt{s})$, where $\wt{\chi}$ lies in the series $\mathcal{E}(\wt{G}, \wt{s})$.  Hence when $q\equiv 3\pmod 4$, it remains to note that $\wt{\delta}^\tau=\wt{\delta}$ if and only if $\wt{\delta}_1^\tau=\wt{\delta}_1$, since they are linear and must have the same order.
\end{proof}

\def\cprime{$'$} \def\cprime{$'$}

\end{document}